\documentclass[12pt]{amsart}
\usepackage[latin1]{inputenc}
\usepackage[english]{babel}
\usepackage{amsmath,amssymb,amsthm,amsfonts, color}
\usepackage{enumerate}
\usepackage{graphicx}
\usepackage{latexsym}
\usepackage[all]{xy}
\usepackage{tikz}

\makeatletter
\@namedef{subjclassname@2010}{ \textup{2010} Mathematics Subject Classification}
\makeatother

\newtheorem{theorem}{Theorem}

\newtheorem{remark}{Remark}

\theoremstyle{definition}

\begin{document}

\title{On one class of Orlicz functions}

\author{Sergey V. Astashkin}
\address{Department of Mathematics,
Samara National Research University, Moskovskoye shosse 34, 443086,
Samara, Russia}
\email{astash56@mail.ru}
\thanks{$^*$The work was supported by the Ministry of Science and Higher Education of the Russian Federation,
project 1.470.2016/1.4 and by the RFBR grant 18-01-00414.}

\date{January 22, 2016 and, in revised form, January 22, 2016.}


\begin{abstract}
Answering to a recent question raised by Le\'{s}nik, Maligranda, and Tomaszewski, we prove that there is an Orlicz function $\Phi$ with the upper Matuszewska-Orlicz index equal to $1$ such that the Orlicz space $L_\Phi$ does not satisfy Dunford-Pettis criterion of weak compactness.
\end{abstract}

\maketitle

\markright{On one class of Orlicz functions}


Recall that a set $K\subset L_1:=L_1([0,1],\mu)$, where $\mu$ is the usual Lebesgue measure, is called {\it equi-integrable} (or {\it uniformly integrable}) if
$$
\lim_{\delta\to 0}\sup_{\mu(E)<\delta}\sup_{x\in K}\int_E |x(t)|\,dt=0.$$
The following famous description of relatively weakly compact subsets in $L_1(\mu)$ as equi-integrable sets is due to N.~Dunford and B.~J.~Pettis (see \cite{DF-1940} or \cite[Theorem~5.2.9]{AK}).

\begin{theorem}\label{Th D-P}
Let $K\subset L_1$. The following conditions are equivalent:

(i) $K$ is relatively weakly compact in $L_1$;

(ii) $K$ is equi-integrable.
\end{theorem}

The notion of equi-integrability can be easily generalized to rearrangement invariant (r.i.) spaces $X$ on $[0,1]$ (for the definition and properties of these spaces, we refer the reader to the monographs \cite{KPS-IOLO,LT-II}). We shall say that a set $K\subset X$ has {\it equi-absolutely continuous norms} in $X$ if
$$
\lim_{\delta\to 0}\sup_{\mu(E)<\delta}\sup_{x\in K}\|x\chi_{E}\|_X=0.$$
An r.i. space $X$ is said to satisfy {\it Dunford-Pettis criterion of weak compactness (shortly $X\in (WDP)$)} if every relatively weakly compact subset of $X$ has equi-absolutely continuous norms in $X$. Observe that the converse assertion holds in every r.i. space.

In \cite{AKS}, it has been obtained a characterization of r.i. spaces satisfying Dunford-Pettis criterion of weak compactness.
Moreover, a special criterion has been proved there for Orlicz spaces. To state it, we need some definitions (for more detailed information on Orlicz spaces, see the monographs \cite{KrRu1958,LT-I,RaRe1991}).

Let $F$ be an Orlicz function, i.e., an increasing convex function on $[0,\infty)$ such that $F(0)=0.$
Denote by $L_F$ the Orlicz space on $[0,1]$ endowed with the Luxemburg norm
\[\|x\|_{L_F}:=\inf\{\lambda >0:\;\int\limits_{0}^{1}
F(|x(t)|/\lambda)dt\leq 1\}.\]
In particular, if $F(u)=u^p$, $1\le p<\infty$, we obtain the space $L_p$. An Orlicz space $L_F$ is separable if and only if the function $F$ satisfies the {\it $\Delta_2$-condition at infinity}, i.e., there are $C>0$ and $u_0>0$ such that $F(2u)\le CF(u)$ for all $u\ge u_0$. 

Denote by $\nabla_3$ the class of all Orlicz functions $F$ such that 
$$\lim_{t\to \infty}\frac
{\tilde{F}(Ct)}{\tilde{F}(t)}=\infty$$ 
for some $C>1$, where $\tilde{F}$ is the complementary function to $F$
defined by
$$
\tilde{F}(t):=\sup\{ts-F(s):\,s\ge 0\},\;\;t\ge 0.$$

The following theorem is a combination of \cite[Proposition~5.8]{AKS} and \cite[Theorem~3.4]{A-19} (see also Proposition~4.9 in  \cite{FHSTT}). 

\begin{theorem}
  \label{Th1}
Let $F$ be an Orlicz function on $[0,\infty)$. The following conditions are equivalent:

(a) $L_F\in (WDP)$; 

(b) either $L_F=L_1$, or $F\in \nabla_3$;  

(c) each normalized sequence of pairwise disjoint functions from $L_F$ contains a subsequence equivalent to the unit vector basis in $l_1$.

\end{theorem}

\begin{remark}
This result extends also to the Orlicz-Lorentz spaces $L_{F,w}$, where $F$ is an Orlicz function and $w$ is an increasing nonnegative function on $[0,1]$, $w\in L_1$ \cite{AS-20}.
\end{remark}

To clarify the condition (c) of Theorem~\ref{Th1}, we proceed with some more definitions. 

An r.i. space $X$ is called {\it disjointly homogeneous} (shortly DH) if two arbitrary normalized disjoint sequences in $X$ contain equivalent subsequences. In particular, given $1\le p\le\infty$, an r.i. space $X$ is  {\it $p$-disjointly homogeneous} (shortly p-DH) if each  normalized disjoint sequence in $X$ has a subsequence equivalent to the unit vector basis of $l_p$ ($c_0$ when $p=\infty$). These notions were first introduced in \cite{FTT-09} and proved to be very useful in studying the general problem of identifying r.i. spaces $X$ such that the ideals of strictly singular and compact operators bounded in $X$ coincide \cite{bib:04} (see also survey \cite{FHT-survey} and references therein).

Let $M$ be an Orlicz function. We define the following subsets of the space $C[0, \frac{1}{2}]$ of continuous functions on $[0, \frac{1}{2}]$:

$$
E_{M, A}^{\infty} := \overline{\big\{ G(x) = \frac{M(xy)}{M(y)} \ : y > A > 0 \big\}},
$$
$$
E_{M}^{\infty}: = \bigcap_{A > 0}E_{M, A}^{\infty}, \ \ C_{M}^{\infty}: = \overline{conv E_{M}^{\infty}},
$$
where the closure is taken in $C[0, \frac{1}{2}].$ All these sets are nonempty and compact in $C[0, \frac{1}{2}]$  \cite[Lemma~4.a.6]{LT-I}. It is well known that they largely determine properties of sequences of pairwise disjoint functions in Orlicz spaces on $[0,1]$ (see \cite[\S\,4.a.]{LT-I}, \cite{LTIII}).

The following characterization of DH Orlicz spaces has been proved in \cite{bib:04} (see Theorem~4.1). We write $E_{M}^{\infty}\cong \{ F\}$ if all functions from the set $E_{M}^{\infty}$ are equivalent at zero to the function $F$. 

\begin{theorem}
  \label{Th2}
Let $M$ be an Orlicz function on $[0,\infty)$. The Orlicz space $L_M$ is DH if and only if there is a function $F$ such that $E_{M}^{\infty}\cong \{ F\}$. Moreover, $L_M$ is $DH$ if and only if it is $p$-DH for some $1\leq p \leq \infty$. Then, $E_{M}^{\infty} \cong \{ t^{p} \}$ for $1 \leq p < \infty$, and $E_{M}^{\infty} \cong \{ F_{0} \}$ for $p = \infty$, where $F_{0}$ is a generate Orlicz function.
\end{theorem}

Clearly, the condition $M\in \nabla_3$ means that $M(t)$ is in a sense close to the function $H(t):=t$ (resp. the Orlicz space $L_M$ is  located "close"\:to $L_1$). In a different way, the closeness of $M$ to $H$ can be expressed by using the well-known Matuszewska-Orlicz indices $\alpha_{M}^{\infty}$ and $\beta_{M}^{\infty}$ defined by 
$$
\alpha_{M}^{\infty}: = \sup \big\{ p : \sup_{x, y \geq 1} \frac{M(x)y^{p}}{M(xy)} < \infty \big\}, \ \ \ \beta_{M}^{\infty}: = \inf \big\{ p : \inf_{x, y \geq 1} \frac{M(x)y^{p}}{M(xy)} > 0 \big\}.
$$
One can easily check that $1 \leq \alpha_{M}^{\infty} \leq \beta_{M}^{\infty} \leq \infty$. Moreover, $t^{p} \in C_{M}^{\infty},$ where $1 \leq p < \infty$, if and only if $p \in [\alpha_{M}^{\infty}, \beta_{M}^{\infty}]$ (see \cite[Proposition~5.3]{KamRy}). 

It can be easily proved that from the condition $\varphi \in \nabla_{3}$ it follows that $\beta_{\varphi}^{\infty} = 1$ (see \cite[Proposition~10]{LesnikMT}). For the reader's convenience, we present here a very simple argument showing this. 

Assume that $\varphi \in \nabla_{3}$. Applying Theorems~\ref{Th1} and \ref{Th2}, we have $E_{\varphi}^{\infty}\cong \{ t\}$, and hence $C_{\varphi}^{\infty}\cong \{ t\}$. Therefore, $C_{\varphi}^{\infty}$ does not contain the functions $t^{p}$ for $p \neq 1.$ Consequently, according to the above-mentioned result \cite[Proposition~5.3]{KamRy}, the interval $[\alpha_{\varphi}^{\infty}, \beta_{\varphi}^{\infty}]$ consists only of one point, i.e., $\alpha_{\varphi}^{\infty} = \beta_{\varphi}^{\infty} = 1$.  

Recently, in \cite{LesnikMT}, Le\'{s}nik, Maligranda, and Tomaszewski, using somewhat different notation, asked whether the converse holds, i.e., does $\beta_{\varphi}^{\infty} = 1$ imply $\varphi \in \nabla_{3}$ ? The main purpose of this note is to give the negative answer to this question. 

\begin{theorem}\label{120120205}
There exists an Orlicz function $\Phi$ such that $\beta_{\Phi}^{\infty} = 1$ but $\Phi \notin \nabla_{3}.$ 
\end{theorem}
\begin{proof}

First of all, we observe that the theorem will be proved once we construct an Orlicz function $\Phi$ with the properties that $\beta_{\Phi}^{\infty} = 1$ and the set $E_{\Phi}^{\infty}$ contains a function, not equivalent at zero to the function $H$. Indeed, then by Theorem~\ref{Th2} the Orlicz space $L_{\Phi}$ fails to be a 1-DH space, and so from Theorem~\ref{Th1} it follows that $\Phi \notin \nabla_{3}$.

Let us introduce some auxiliary functions. We set $\phi(t)= 1$ if $0 < t \leq 4,$ and for $n = 3, 4 \dots$ 
\begin{equation*}
\phi(t):= 
\begin{cases}
1 \ \ \ \text{ if } 2^{n-1} < t \leq 2^{n} - 2^{\sqrt n}\\
2 \ \ \ \text{ if } 2^{n} - 2^{\sqrt n} < t \leq 2^{n}.
\end{cases}
\end{equation*}
Furthermore, $f(x) := \int\limits_{0}^{x}\phi(t)\,dt$, $x \geq 0$, and
\begin{equation*}
F(x):= 
\begin{cases}
2^{f(\log_{2}x)} \ \ \ \text{ for } x \geq 1 \\
x \ \ \ \ \ \ \ \ \ \ \text{ for } 0 < x <1.
\end{cases}
\end{equation*}
Then, the function $\Phi$ that we need is defined by
$$
\Phi(x) := \int\limits_{0}^{x}\frac{F(t)}{t}\,dt.
$$
Let us check that $\Phi$ is an Orlicz function on $[0,\infty)$. One can easily see that it suffices to show that the function ${F(y)}/{y}$ increases for $y > 0$, or equivalently 
\begin{equation}\label{24112019}
\frac{F(xy)}{F(y)} \geq x\;\;\mbox{for all}\;\; x \geq 1, y > 0.
\end{equation}
Suppose first that $y \geq 1.$ From the definitions of the functions $f$ and $F$ it follows 
$$
\frac{F(xy)}{F(y)} = 2^{f(\log_{2}x + \log_{2}y) - f(\log_{2}y)}
$$
and
$$
f(\log_{2}x + \log_{2}y) - f(\log_{2}y) = \int\limits_{\log_{2}y}^{\log_{2}x + \log_{2}y}\phi(t)dt \geq \log_{2}x.
$$
Hence, $\eqref{24112019}$ is proved for $y\ge 1$. Let now  $0 < y < 1$. Then, if $xy < 1$, we have $F(xy) = xy = xF(y)$. Otherwise, in the case when $xy > 1$,  
$$
f(\log_{2}(xy)) = \int\limits_{0}^{\log_{2}(xy)}\phi(t)dt \geq \log_{2}(xy),
$$
whence
$$
F(xy) = 2^{f(\log_{2}(xy))} \geq xy = xF(y).
$$
Thus, inequality $\eqref{24112019}$ is established, and so  $\Phi$ is an Orlicz function.

The next our goal is to prove that the upper Matuszewska-Orlicz index $\beta_{\Phi}^{\infty}$ is equal to $1.$ To this end, we are going to show that for every $p > 1$ there is a constant $C_{p} > 0$ such that for all $x, y \geq 1$
\begin{equation}\label{241120191}
\frac{F(xy)}{F(y)} \leq C_{p}x^{p}.
\end{equation}

Note that it suffices to check inequality \eqref{241120191} only for "large"\:values of $x$ and $y$. In fact, let us assume that \eqref{241120191} is already proved in the case when $x,y \ge K > 1$. Then, for all $x,y\ge 1$
$$
f(\log_{2}y + \log_{2}x) - f(\log_{2}y) = \int\limits_{\log_{2}y}^{\log_{2}x + \log_{2}y}\phi(t)\,dt \leq 2\log_{2}x = \log_{2}x^{2},
$$
and, if additionally $1 \leq x \leq K$, we have  
$$
\frac{F(xy)}{F(y)} \leq 2^{\log_{2}x^{2}} = x^{2} \leq K^{2}x^{p}.
$$
In the remaining case when $x \geq K$ and $1 \leq y \leq K$, taking into account that $F(y)$ increases, $F(0)=1$ and using the hypothesis, we get
$$
\frac{F(xy)}{F(y)} \leq F(K) \frac{F(xK)}{F(K)} \leq F(K)C_{p}x^{p}.
$$
As a result, inequality \eqref{241120191} holds for all $x, y \geq 1$ (possibly with a larger constant $C_p$). Next, we shall prove \eqref{241120191} for "large"\:values of $x$ and $y$.

Let $0<\varepsilon \le (p-1)/{2}$ be fixed. It is easy to see that there exists $n_0\in\mathbb{N}$ such that for all  $n, m \in \mathbb{N}$, $m-1>n\ge n_0$ the following inequality holds:
\begin{equation}\label{27112019}
\sum_{i=n}^{m-1}2^{\sqrt i} < \varepsilon \sum_{i=n+1}^{m-1}\big(2^{i-1} - 2^{\sqrt i}\big),
\end{equation}
Then, assuming that $x,y\ge 2^{2^{n_0}}$, we can find $n, m \in \mathbb{N}$, $m-1>n\ge n_0$, satisfying the inequalities $2^{n-1} < \log_{2}y \leq 2^{n},$ $2^{m-1} < \log_{2}y + \log_{2}x \leq 2^{m}$. In particular, for the chosen $n$ and $m$ we have \eqref{27112019}.

We set
$$
A: = \big \{ t \in [\log_{2}y, \log_{2}y + \log_{2}x]: \phi(t) = 2 \big\}
$$
and
$$
B := \big \{ t \in [\log_{2}y, \log_{2}y + \log_{2}x]: \phi(t) = 1 \big\}.
$$
There are four possible different arrangements of the numbers $\log_{2}y$ and $\log_{2}y + \log_{2}x$ inside the intervals $[2^{n-1}, 2^{n}]$ and $[2^{m-1}, 2^{m}]$, respectively. Let us consider them successively.

If $2^{n-1} < \log_{2} y \leq 2^{n} - 2^{\sqrt n}$ and $2^{m} - 2^{\sqrt m} \leq \log_{2}y + \log_{2}x \leq 2^{m}$, then by the definition of $\phi$ we get
$$
\mu(A) = \sum_{i=n}^{m-1}2^{\sqrt i} + (\log_{2}y + \log_{2}x - 2^{m} + 2^{\sqrt m}) \leq \sum_{i = n}^{m}2^{\sqrt i},
$$
while
$$
\mu(B) = (2^{n} - 2^{\sqrt n} - \log_{2}y) + \sum_{i = n+1}^{m}(2^{i -1} - 2^{\sqrt i}) \geq \sum_{i = n+1}^{m}(2^{i-1} - 2^{\sqrt i}).
$$
In the case when $2^{n} - 2^{\sqrt n} < \log_{2} y < 2^{n}$ and $2^{m-1} < \log_{2}x + \log_{2}y<2^{m} - 2^{\sqrt m}$, we have
$$
\mu(A)= (2^{n} - \log_{2}y) +\sum_{i=n+1}^{m-1}2^{\sqrt i}\leq \sum_{i=n}^{m-1}2^{\sqrt i}
$$
and
$$
\mu(B) =  \sum_{i = n+1}^{m-1} (2^{i-1} - 2^{\sqrt i})+ (2^{m} - 2^{\sqrt m}-\log_{2}x - \log_{2}y )\geq \sum_{i = n+1}^{m-1} (2^{i-1} - 2^{\sqrt i}).
$$
Similarly, if $2^{n} - 2^{\sqrt n} < \log_{2} y < 2^{n}$ and $2^{m} - 2^{\sqrt m} \leq \log_{2}y + \log_{2}x \leq 2^{m}$, 
$$
\mu(A)\leq \sum_{i = n}^{m}2^{\sqrt i}\;\;\mbox{and}\;\;\mu(B)\geq \sum_{i = n+1}^{m}(2^{i-1} - 2^{\sqrt i}),
$$ 
and if $2^{n-1} < \log_{2} y \leq 2^{n} - 2^{\sqrt n}$ and $2^{m-1} < \log_{2}x + \log_{2}y<2^{m} - 2^{\sqrt m}$, then
$$
\mu(A)= \sum_{i = n}^{m-1}2^{\sqrt i}\;\;\mbox{and}\;\;\mu(B)\geq \sum_{i = n+1}^{m-1}(2^{i-1} - 2^{\sqrt i}).
$$

From the estimates obtained and inequality $\eqref{27112019}$ it follows that
\begin{equation}\label{271120191}
\mu(A) < \varepsilon \mu(B).
\end{equation}
Since $\mu(B) \le \log_{2}x$, then combining \eqref{271120191} with the definition of $\phi$, we get
\begin{eqnarray*}
f(\log_{2}y + \log_{2}x) - f(\log_{2}y) 
&=& 
\int\limits_{\log_{2}y}^{\log_{y} + \log_{2}x}\phi(t)\,dt = \int\limits_{A}\phi(t)\,dt + \int\limits_{B}\phi(t)\,dt
\\ &\leq& 2\varepsilon \mu(B) + \mu(B) \leq (1 + 2\varepsilon)\log_{2}x,
\end{eqnarray*}
and hence, by the choice of $\varepsilon$,
$$
\frac{F(xy)}{F(y)}= 2^{f(\log_{2}x + \log_{2}y) - f(\log_{2}y)} \leq 2^{(1+2\varepsilon)\log_{2}x} = x^{1 + 2\varepsilon} \leq x^{p}
$$
for all $x,y\ge 2^{2^{n_0}}$. Thus, taking into account the above observation, we obtain $\eqref{241120191}$ for all $x,y\ge 1$.

Further, let us check that for some $c > 0$ we have
\begin{equation}\label{271120194}
c F(x) \leq \Phi(x) \leq F(x),\;\;x>0.
\end{equation}

The right-hand side inequality in \eqref{271120194} is an immediate consequence of the definition of  $\Phi$ and the fact that ${F(y)}/{y}$ is an increasing function for $y > 0$. On the other hand, 
\begin{equation*}\label{241120193}
\Phi(x) \geq \int\limits_{{x}/{2}}^{x}{F(t)}/{t}\,dt \geq F({x}/{2}).
\end{equation*}
Moreover, from inequality \eqref{241120191} for $x = 2$ and $y \geq 1$ it follows that
$$
F(x) \leq C_{p}2^{p}F({x}/{2}),\;\;x\ge 2.
$$
Since $F(x) = x$ if $0<x< 16$ and $C_p\ge 1$, then this inequality holds for all $x>0$. Combining the last estimates, we arrive at the left-hand side inequality in \eqref{271120194}.

It is easy to see that from inequalities \eqref{271120194}, \eqref{241120191}, and the definition of the upper Matuszewska-Orlicz index it follows that $\beta_{\Phi}^{\infty} = 1.$

It remains to show that the set $E_{\Phi}^{\infty}$ contains a function that is not equivalent at zero to the function $H(t)=t$. 

Let $m \in \mathbb{N}$ be fixed. Then we have $
2^{n} - 2^{\sqrt n} < 2^{n}- 2^{m}$ whenever $n > m^{2}$. Hence, $\phi(t) =2$ for all $2^{n} - 2^{m}\le t\le 2^{n}$. Consequently,
$$
\log_{2}\left(\frac{F(2^{-2^{m}} \cdot 2^{2^{n}})}{F(2^{2^{n}})}\right)
= f(2^{n} - 2^{m}) - f(2^{n}) = - \int\limits_{2^{n} - 2^{m}}^{2^{n}}\phi(t)\,dt =- 2 \cdot 2^{m},
$$
whence
\begin{equation}\label{271120192}
\frac{F(2^{-2^{m}}\cdot2^{2^{n}})}{F(2^{2^{n}})} = \Big( 2^{-2^{m}} \Big)^{2}.
\end{equation}

From inequality $\eqref{271120194}$ and equation $\eqref{271120192}$ we get the following estimate for the points $t_{n}: = 2^{2^{n}}$, $n \geq m^{2}$:
\begin{equation}\label{271120195}
\frac{\Phi(2^{-2^{m}}t_{n})}{\Phi(t_{n})} \leq c^{-1}\frac{F(2^{-2^{m}}t_{n})}{F(t_{n})} = c^{-1}\Big( 2^{-2^{m}} \Big)^{2}.
\end{equation}
Consider now the functions 
$$
\Phi_{n} (x) := \frac{\Phi(xt_{n})}{\Phi(t_{n})}, n = 1, 2, \dots,
$$
Clearly, $\Phi_{n}\in E_{\Phi,A_m}^{\infty}$, $n\ge m^2$, where $A_m:= 2^{2^{m^2}}$. Since $E_{\Phi, A_m}^{\infty}$ is a compact set in $C[0,\frac12]$, there exists an increasing subsequence of positive integers $\{ n_{k} \}_{k=1}^{\infty}$ such that $\Phi_{n_{k}} (x) $ uniformly converges on $[0,\frac12]$ to some function $N(x).$ One can easily to see that $N$ belongs to the set $E_{\Phi,A}^{\infty}$ for each $A>0$, which implies that $N \in E_{\phi}^{\infty}$. In addition, in view of inequality \eqref{271120195}, we have for all $m = 1, 2, \dots$
$$
N(2^{-2^{m}}) \leq c^{-1}\Big( 2^{-2^{m}} \Big)^{2}.
$$
It is immediately follows from this estimate that $N(t)$ is not equivalent at zero to the function $H$. Summarizing everything, we complete the proof. 
\end{proof}

\end{document}